\documentclass[twosided]{article}
  \usepackage{amsmath,amssymb,amsthm,amsfonts}
  \usepackage{hyperref}
  \usepackage{mathtools}
  \usepackage{latexsym}
  \usepackage{mathrsfs}
  \usepackage{pstricks}
   \usepackage{tikz-cd}
  \usepackage{epsfig}
  \usepackage{verbatim}
  \usepackage{float}
  \usepackage{booktabs}
  \usepackage{enumitem}
  \usepackage{textcomp}
  \usepackage{tikz}
  \usetikzlibrary{matrix,arrows,decorations.pathmorphing}

  \newcommand{\wt}{\widetilde}
  \newcommand{\gdeg}{G\text{\rm -deg}}

  \newcommand{\vp}{\varphi}
  \newcommand{\ve}{\varepsilon}

  \DeclareMathOperator{\id}{Id}
  \def\bn{\mathbb N}
\def\br{\mathbb R}
\def\bz{\mathbb Z}
\def\mbdeg{\mbox{deg}}
  \newcommand{\vs}{\vskip .3cm}
  
  \newcommand{\amal}[5]{#1\prescript{#2}{}\times_{#3}^{#4}#5}

  \newcommand{\norm}[1]{\left\lVert#1\right\rVert}

  \newcommand\cV{\ensuremath{\mathcal V}}
  
  \newcommand\cW{\ensuremath{\mathcal W}}

 \newcommand\ff{\ensuremath{\mathfrak f}}

  \newcommand\bbR{\ensuremath{\mathbb R}}

  \newcommand\bbV{\ensuremath{\mathbb V}}

  \newcommand\bbZ{\ensuremath{\mathbb Z}}

  \newcommand\bfx{\ensuremath{\mathbf x}}

  \newcommand\bfV{\ensuremath{\mathbf V}}

  \newcommand\scrA{\ensuremath{\mathscr A}}

  \newcommand\scrE{\ensuremath{\mathscr E}}
  \newcommand\scrF{\ensuremath{\mathscr F}}
  
  \newcommand\scrH{\ensuremath{\mathscr H}}

  \newcommand\itv{\ensuremath{\mathit v}}
  \newrgbcolor{violet}{.6 .1 .8}
  \newrgbcolor{lightyellow}{1 1 .8}
  \newrgbcolor{lightblue}{.80 1 1}
  \newrgbcolor{mygreen}{0 .66 .05}
  \definecolor{mygreen}{rgb}{0,.66,.05}
  \definecolor{lightyellow}{rgb}{1,1,.80}
  \newrgbcolor{orange}{1 .6 0}
  \newrgbcolor{GreenYellow}{.85 1 .31}
  \newrgbcolor{Yellow}{1  1  0}
  \newrgbcolor{Goldenrod}{1  .90  .16}
  \newrgbcolor{Dandelion}{1  .71  .16}
  \newrgbcolor{Apricot}{1  .68  .48}
  \newrgbcolor{Peach}{1  .50  .30}
  \newrgbcolor{Melon}{1  .54  .50}
  \newrgbcolor{YellowOrange}{1  .58  0}
  \newrgbcolor{Orange}{1  .39  .13}
  \newrgbcolor{BurntOrange}{1  .49  0}
  \newrgbcolor{Bittersweet}{1.  .4300  .24}
  \newrgbcolor{RedOrange}{1  .23  .13}
  \newrgbcolor{Mahogany}{1.  .4475  .4345}
  \newrgbcolor{Maroon}{1.  .4084  .5376}
  \newrgbcolor{BrickRed}{1.  .3592  .3232}
  \newrgbcolor{Red}{1  0  0}
  \newrgbcolor{OrangeRed}{1  0  .50}
  \newrgbcolor{RubineRed}{1  0  .87}
  \newrgbcolor{WildStrawberry}{1  .04  .61}
  \newrgbcolor{CarnationPink}{1  .37  1}
  \newrgbcolor{Salmon}{1  .47  .62}
  \newrgbcolor{Magenta}{1  0  1}
  \newrgbcolor{VioletRed}{1  .19  1}
  \newrgbcolor{Rhodamine}{1  .18  1}
  \newrgbcolor{Mulberry}{.6668  .1180  1.}
  \newrgbcolor{RedViolet}{.9538  .4060  1.}
  \newrgbcolor{Fuchsia}{.5676  .1628  1.}
  \newrgbcolor{Lavender}{1  .52  1}
  \newrgbcolor{Thistle}{.88  .41  1}
  \newrgbcolor{Orchid}{.68  .36  1}
  \newrgbcolor{DarkOrchid}{.60  .20  .80}
  \newrgbcolor{Purple}{.55  .14  1}
  \newrgbcolor{Plum}{.50  0  1}
  \newrgbcolor{Violet}{.98 .15 .95}
  \newrgbcolor{RoyalPurple}{.25  .10  1}
  \newrgbcolor{BlueViolet}{.84  .38  .98}
  \newrgbcolor{Periwinkle}{.43  .45  1}
  \newrgbcolor{CadetBlue}{.38  .43  .77}
  \newrgbcolor{CornflowerBlue}{.35  .87  1}
  \newrgbcolor{MidnightBlue}{.4414  .9259  1.}
  \newrgbcolor{NavyBlue}{.06  .46  1}
  \newrgbcolor{RoyalBlue}{0  .50  1}
  \newrgbcolor{Blue}{0  0  1}
  \newrgbcolor{Cerulean}{.06  .89  1}
  \newrgbcolor{Cyan}{0  1  1}
  \newrgbcolor{ProcessBlue}{.04  1  1}
  \newrgbcolor{SkyBlue}{.38  1  .88}
  \newrgbcolor{Turquoise}{.15  1  .80}
  \newrgbcolor{TealBlue}{.1572  1.  .6668}
  \newrgbcolor{Aquamarine}{.18  1  .70}
  \newrgbcolor{BlueGreen}{.15  1  .67}
  \newrgbcolor{Emerald}{0  1  .50}
  \newrgbcolor{JungleGreen}{.01  1  .48}
  \newrgbcolor{SeaGreen}{.31  1  .50}
  \newrgbcolor{Green}{0  1  0}
  \newrgbcolor{ForestGreen}{.1992  1.  .2256}
  \newrgbcolor{PineGreen}{.3100  1.  .5575}
  \newrgbcolor{LimeGreen}{.50  1  0}
  \newrgbcolor{YellowGreen}{.56  1  .26}
  \newrgbcolor{SpringGreen}{.74  1  .24}
  \newrgbcolor{OliveGreen}{.6160  1.  .4300}
  \newrgbcolor{RawSienna}{.53  .28  .16}
  \newrgbcolor{Sepia}{1.  .7510  .70}
  \newrgbcolor{Brown}{.41  .25  .18}
  \newrgbcolor{TAN}{.86  .58  .44}
  \newrgbcolor{Gray}{1.  1.  1.}
  \newrgbcolor{Black}{1  1  1}
  \newrgbcolor{White}{1  1  1}

  \theoremstyle{plain}
  \newtheorem{theorem}{Theorem}[section]
  \newtheorem{proposition}[theorem]{Proposition}
  \newtheorem{lemma}[theorem]{Lemma}

  \theoremstyle{definition}
  
  \newtheorem{definition}[theorem]{Definition}
  \newtheorem{remark}[theorem]{Remark}

  \newcommand{\TheTitle}{Periodic Solutions to Reversible Second Order Autonomous Systems with Commensurate Delays}
\newcommand{\support}{This project is supported by the National Natural
Science Foundation of China (No.~11871171).}
\newcommand{\corresponding}{Corresponding author}
\newcommand{\AuthorA}{Zalman Balanov}
\newcommand{\AuthorB}{Fulai Chen\thanks{\corresponding} }
\newcommand{\AuthorC}{Jing Guo}
\newcommand{\AuthorD}{Wies{\l}aw Krawcewicz\thanks{\support}~}
\newcommand{\InstituteA}{Department of Mathematics, Xiangnan University, Xiangnan, Hunan, China, and  
  Department of Mathematical Sciences, the University of Texas at Dallas, Richardson, Texas 75080, USA.}
  \newcommand{\InstituteB}{Department of Mathematics, Xiangnan University, Chenzhou,, Hunan, China.}
    \newcommand{\InstituteC}{ Department of Mathematical Sciences, the University of Texas at Dallas, Richardson, Texas 75080, USA.}
  \newcommand{\InstituteD}{Applied Mathematics Center at Guangzhou University,
    Guangzhou 510006, China, and
    Department of Mathematical Sciences, the University of Texas at Dallas,
    Richardson, Texas 75080, USA.}
\newcommand{\EmailA}{\href{mailto:balanov@utdallas.edu}
{~~\tt balanov@utdallas.edu}}
\newcommand{\EmailB}{\href{mailto:cflmath@163.com}
{~~\tt cflmath@163.com}}
\newcommand{\EmailC}{\href{mailto:jig70@pitt.edu}
{~~\tt jig70@pitt.edu}}      
\newcommand{\EmailD}{\href{mailto:wieslaw@utallas.edu}
{~~\tt wieslaw@utallas.edu}}

  \title{\TheTitle}
  \author{
    \AuthorA\thanks{\InstituteA \EmailA},
    \AuthorB\thanks{\InstituteB \EmailB}, 
    \AuthorC\thanks{\InstituteC \EmailC} ~ and 
    \AuthorD\thanks{\InstituteD \EmailD}
   }
  \date{\today}

  \hypersetup{
    unicode=true,
    pdfauthor={\AuthorA, \AuthorB, \AuthorC,~and ~\AuthorD},
    pdftitle={\TheTitle},
    pdfsubject={delay system},
    pdfkeywords={}
  }

\begin{document}

  \maketitle
  
  \begin{abstract}
  Existence and spatio-temporal patterns of periodic solutions to second order reversible equivariant autonomous systems with commensurate delays
  are studied using the Brouwer $O(2) \times \Gamma \times \mathbb Z_2$-equivariant degree theory, where
  $O(2)$ is related to the reversing symmetry, $\Gamma$ reflects the
  symmetric character of the coupling in the corresponding network and $\mathbb Z_2$ is related to the oddness of the right-hand-side. Abstract results are supported by a concrete example with $\Gamma = D_6$ -- the dihedral group of order 12. 
  \end{abstract}


  	
  \section{Introduction}\label{sec:introduction}
  
\
%\noindent

{\bf (a) Subject and goal.}   Existence of periodic solutions to equivariant dynamical systems together with  describing their spatio-temporal symmetries constitute an important problem of equivariant dynamics (see, for example, \cite{GolSchSt,GolStew} for the equivariant singularity theory based methods and \cite{AED,survey,IV-book} for the equivariant degree treatment). As is well-known, second order systems of ODEs with no friction term exhibit an extra symmetry -- the so-called reversing symmetry, i.e. if $x(t)$ is a solution to the system, then so is $x(-t)$. We refer to \cite{Lamb-Roberts} for a comprehensive exposition of (equivariant) reversible ODEs as well as their applications in natural sciences (see also \cite{AS}). It should be stressed that in the context relevant to spatio-temporal symmetries of periodic solutions, reversing symmetry gives rise to extra subgroups of the non-abelian group $O(2)$.

Simple examples show that, in contrast to their ODEs counterparts, second order delay differential equations (in short, DDEs) with no friction term are not reversible, in general. In \cite{BW} (see also \cite{KW}), we considered  {\it space reversible} equivariant mixed DDEs of the form
\begin{align}
\label{eq:dde_system}
\ddot v(y)=g(\alpha, v(y))+a(v(y-\alpha)+v(y+\alpha)), \quad a,\alpha \in\br ,
\end{align}
 with equivariant $g : \mathbb R^n \to \mathbb R^n$ 
(one can think of equations governing steady-state
solutions to PDEs, cf.~\cite{Lamb-Roberts} and references therein). Note that by replacing $y$ by $t$ in \eqref{eq:dde_system}, 
one obtains {\it  time-reversible} DDEs. However, such systems involve using the information from the future by ``traveling back in time'', which is 
difficult to justify from a commonsensical viewpoint.   
  
Time delay systems with {\it commensurate delays} play an important role in robust control theory (see, for example, \cite{GuoKharitonovCheng} and references therein). A class of such systems exhibiting a reversal symmetry is the main {\it subject} of the present paper. To be more specific,  given $p>0$, we are interested in the periodic problem
\begin{align}\label{weq}
\left\{  \begin{aligned}
 \ddot{x}(t)&=f\left(x(t),x\left(t-\frac{p}{m}\right),\dots,x\left(t-(m-1)\frac{p}{m}\right)\right),\quad t\in\bbR,\;x(t)\in\mathbf{V} = \mathbb R^n,\\
 x(t)&=x(t+p),\quad\dot{x}=\dot{x}(t+p)
 \end{aligned}\right.
 \end{align}
 under the following assumption providing the time reversibility of system \eqref{weq}:

\medskip
 \begin{itemize}
\item[(R)] $f(x,y^1,y^2, \cdots, y^{m-2},y^{m-1})=f(x,y^{m-1},y^{m-2}, \cdots, y^2, y^1)$ 
for all  $(x,y^1,\cdots,y^{m-1})\in \bfV^m$
\end{itemize}

Assume, in addition, that $\Gamma$ is a finite group and $\bfV$ is an orthogonal $\Gamma$-representation ($\Gamma$ acts on $\bfV = \bbR^n$ by permuting the vector coordinates in $\bbR^n$). 
Put $\bfx:=(x,y^1,\cdots,y^{m-1})\in \bfV^m$ and define on $\bfV^m$ the diagonal $\Gamma$-action by 
$\gamma \bfx := (\gamma x, \gamma y^1,...,\gamma y^{m-1})$. We make the following assumptions: 
 \begin{enumerate}[label=($A_\arabic*$)]
  	\item\label{c1} $f$ is $\Gamma$-equivariant, i.e., $f$ is continuous and $f(\gamma \bfx)=\gamma f(\bfx)$ for all $\gamma\in\Gamma$ and $\bfx\in \bfV^m$;
  	\item\label{c2} $f$ is odd, i.e., $f(-\bfx)=-f(\bfx)$, for all $\bfx\in \bfV^m$;
	\item\label{c3} $\exists R>0\; \forall \bfx \quad |x|>R, \mid y^j|\leq |x| \Rightarrow x\bullet f(\bfx)>0 $;
	\item\label{c5} The derivative $A:=Df(0)=[A_0,A_1,\dots, A_{m-1}]$ exists and $A_jA_s =A_sA_j$ for $j,s =0,1,\dots, m-1$.
  \end{enumerate}

\noindent
The {\it goal} of the present paper is to study the existence and spatio-temporal properties of solutions to problem \eqref{weq} under the assumptions (R), \ref{c1}--\ref{c5}.

 \vs 

 {\bf (b) Method.} Observe that given an  orthogonal $G$-representation  $V$ (here  $G$ stands for a compact Lie group) and an admissible $G$-pair  $(f,\Omega)$  in $V$
(i.e. $\Omega\subset V$ is an open bounded $G$-invariant set and $f:V\to V$ is a $G$-equivariant map without zeros on $\partial \Omega$), the Brouwer degree $d_H:=\deg(f^H,\Omega^H)$ is 
well-defined for any $H \le  G$ (here $\Omega^H:= \{x \in \Omega\, :\, hx = x\; \forall h \in H\}$ 
and $f^H:= f|_{\Omega^H}$). If for some $H$, one has $d_H\not=0$, then the existence of solutions with symmetry at least $H$ to equation $f(x)=0$ in $\Omega$, can be predicted. Although this approach provides a  way to determine the existence of solutions in $\Omega$, and even to distinguish their different orbit types, nevertheless, it comes at a price of elaborate $H$-fixed-point space computations which can be a rather challenging task.

Our method is based on the usage of the Brouwer equivariant degree theory; 
 for the detailed exposition of this theory, we refer to the monographs
 \cite{AED, KW,IV-book,KB} and survey \cite{survey} (see also \cite{BKLN,BLN}). 
In short, the equivariant degree is a topological tool allowing ``counting'' orbits of solutions to symmetric equations in the same way as the usual Brouwer degree does, but according to their symmetry properties. 

To be more explicit, the equivariant degree $\gdeg(f,\Omega)$ is an element of the free $\bz$-module $A(G)$ generated by the conjugacy classes $(H)$ of subgroups $H$ of $G$ with a finite Weyl group $W(H)$:
\begin{equation}\label{eq:gdeg}
\gdeg(f,\Omega)=\sum_{(H)} n_H\, (H), \quad n_H\in \bz,
\end{equation} 
where the coefficients $n_H$ are given by the following Recurrence Formula
\begin{equation}\label{eq:rec}
n_H=\frac{d_H-\sum_{(L)>(H)} n_L \,n(H,L)\, |W(L)|}{|W(H)|},
\end{equation}
and  $n(H,L)$ denotes the number of subgroups $L'$ in $(L)$ such that $H\le L'$ (see \cite{AED}).  One can immediately recognize a connection 
between the two collections: $\{d_H\}$ and  $\{ n_H\}$, where $H \le  G$ and $W(H)$ is finite. 
As a matter of fact,  $\gdeg(f,\Omega)$ satisfies the standard properties expected from any topological degree.  
However, there is one additional functorial property, which plays a crucial role in computations, namely the {\it multiplicativity property}. In fact, $A(G)$ has a natural structure of a ring  (which is called the {\it Burnside ring} of $G$), where the multiplication  $\cdot:A(G)\times A(G)\to A(G)$   is defined on generators by 
\begin{equation}\label{eq:mult}
(H)\cdot (K)=\sum_{(L)} m_L\, (L) \quad\quad (W(L) \text{ is finite}),
\end{equation}
where the integer $m_{L}$ represents the number of $(L)$-orbits contained in the space $G/H\times G/K$ equipped with the natural diagonal $G$-action.
The multiplicativity property for two admissible $G$-pairs $(f_1,\Omega_1)$ and   $(f_2,\Omega_2)$ means the following equality:
\begin{equation}\label{eq:mult-property}
 \gdeg(f_1\times f_2,\Omega_1\times \Omega_2)=  \gdeg(f_1,\Omega_1)\cdot  \gdeg(f_2,\Omega_2).
\end{equation}
Given a $G$-equivariant linear isomorphism $A : V \to V$, formula \eqref{eq:mult-property} combined with the equivariant spectral decomposition of $A$, reduces the computations of $\gdeg(A,B(V))$ to the computation of the so-called basic degrees $\deg_{\cV_k}$, 
which can  be `prefabricated'
in advance for any group $G$ (here $\deg_{\cV_k}:=\gdeg(-\id, B(\cV_k))$ with $\cV_k$ being an irreducible $ G$-representation and $B(X)$ stands for the unit ball in $X$).  ln many cases, the equivariant degree  based method can be easily assisted
by computer (its usage seems to be unavoidable for large symmetry groups).

  In the present paper, to establish the abstract results on the existence and symmetric properties of periodic solutions,
  we use the  $G$-equivariant Brouwer degree with $G:=O(2)\times \Gamma\times \bz_2$, where
  $O(2)$ is related to the reversing symmetry, $\Gamma$ reflects the
  symmetric character of the coupling in the corresponding network and $\mathbb Z_2$ is related to the oddness of $f$. We also present
  a concrete illustrating example with $\Gamma:= D_6$, where $D_6$ stands for the dihedral group of order 12. Our computations are essentially based
  on new group-theoretical computational algorithms, which were implemented in the specially created  GAP library by Hao-pin Wu (see \cite{Pin}).
  
 \vs

 {\bf (c) Overview.} After the Introduction, the paper is organized as follows. In Section \ref{sec:SODAS}, we establish a priori estimates for solutions to problem \eqref{eq2} in the space $C^2(S^1;\bfV)$. In Section \ref{sec:operator-reform}, we reformulate problem \eqref{eq2} as an 
 $O(2) \times \Gamma \times \mathbb Z_2$-equivariant fixed point problem  in $C^2(S^1;\bfV)$ and present an abstract equivariant degree based result.
 This result can be effectively applied to concrete symmetric systems only if a ``workable" formula for the degree associated can be elaborated. The latter is a subject of Section \ref{sec:degree-computation}, where we combine the multiplicativity property of the equivariant degree with appropriate  equivariant spectral data of the operators involved. Based on that, in Section \ref{sec:main-results-example}, we present our main results (see Theorems \ref{th:main1} and \ref{th:main2}) and an illustrating example with the dihedral group $\Gamma = D_6$.
 We conclude the paper with an Appendix related to the equivariant topology jargon and equivariant degree background. 

  \section{Normalization of Period and A Priori Bounds}\label{sec:SODAS}

  We start with standard steps of the degree theory treatment of (autonomous) dynamical systems: normalization of the period and establishing 
  a priori bounds. 
  \subsection{Normalization of period}
  By substituting $y(t)=x(\frac{pt}{2\pi})$, system \eqref{weq} is transformed to
    \begin{align*}
   \ddot{y}(t)&=\Big(\frac{p}{2\pi}\Big)^2\ddot{x}\Big(\frac{pt}{2\pi}\Big)\\
    &=\Big(\frac{p}{2\pi}\Big)^2  f\left(x\left(\frac{pt}{2\pi}\right),x\left(\frac{p}{2\pi}\left(t-\frac{2\pi}{m}\right)\right),\cdots,x\left(\frac{p}{2\pi}\left(t-\frac{2\pi}{m}(m-1)\right)\right)\right)\\
    &=\Big(\frac{p}{2\pi}\Big)^2  f\left(y(t),y\left(t-\frac{2\pi}{m}\right),\cdots,y\left(t-\frac{2\pi}{m}(m-1)\right)\right).
    \end{align*}
Put  $\tau:=\frac{2\pi}{m}$, so system \eqref{weq} can be written as
  \begin{align}\label{weqspf} 
  \begin{cases}
   \ddot{y}(t)=\alpha^2f(y(t),y(t-\tau),\dots,y(t-(m-1)\tau))\\
   y(0)=y(2\pi),\; \dot y(0)=\dot y(2\pi),
   \end{cases}
  \end{align}
where $\alpha:=\frac{p}{2\pi}$. Notice that $p$-periodic solutions $x(t)$ to system \eqref{weq} are in one-to-one correspondence to $2\pi$-periodic solutions $y(t)$ to system \eqref{weqspf}. In order to simplify the notation, replace $\alpha^2 f$ by $\mathfrak{f}$ and $y$ by $x$,  so one can represent system \eqref{weq} as follows:
\begin{equation}\label{eq2}
 \begin{cases}
    \ddot{x}(t)=\ff(\bfx_t), \;t\in\mathbb{R},\;x(t)\in \bfV,\\
    x(t)=x(t+2\pi),\; \dot{x}(t)=\dot{x}(t+2\pi),
%   \end{aligned}\right.
%  \end{align}
\end{cases}
\end{equation}
where $\bfx_t:=\left(x(t),x(t-2\pi/m),\dots,x(t-(m-1)2\pi/m)\right)$. 

\begin{remark}\label{rem:same-assumptions}
Notice that $\ff:\bfV^m\to \bfV$  satisfies conditions (R), \ref{c1}--\ref{c5} as well. 
\end{remark}

  \subsection{A priori bounds}
  Consider the following modification of system \eqref{eq2}:
  \begin{equation} \label{eq3}
   \begin{cases}
    \ddot{x}(t) =\lambda(\ff(\bfx_t)-x(t))+x(t),\quad t\in\bbR,\;\; x(t)\in \bfV,\;\;\lambda\in [0,1],\\
    x(t) =x(t+2\pi),\dot{x}=\dot{x}(t+2\pi),
   \end{cases}
  \end{equation}
where $\ff: \bfV^m\rightarrow \bfV$ is a continuous map. 
One has the following
  
\begin{lemma}\label{lm31}
Assume that $\ff : \bfV^m\to \bfV$ satisfies  \ref{c3}. If  $x:\br\to \bfV$ is a $C^2$-differentiable $2\pi$-periodic function such that  
$\displaystyle \max_{t\in\bbR}  |x(t)|>R$, then $x(t)$ is not a solution to \eqref{eq3} for $\lambda\in [0,1]$.
\end{lemma}
 
\begin{proof}
 Assume for the contradiction that $x(t)$ is a solution to \eqref{eq3} with $\displaystyle |x(t_o)|=\max_{t\in\bbR}|x(t)|>R$, and  consider the function $\phi(t):=\frac{1}{2}|x(t)|^2$.  Then, $\displaystyle \phi(t_0)=\max_{t \in\bbR}\phi(t)$, $\phi'(t_0)=x(t_0) \bullet \dot{x}(t_0)=0$ and $\phi''(t_0)=\dot{x}(t_0)\bullet\dot{x}(t_0)+\ddot{x}(t_0)\bullet x(t_0)\leq0$. However, by condition \ref{c3}, one has for $1\ge \lambda>0$:
   \begin{align*}
    \phi''(t_0)&=x(t_0)\bullet\ddot{x}(t_0)+\dot{x}(t_0)\bullet\dot{x}(t_0)\\
               &\geq \lambda x(t_0)\bullet \ff\Big(x(t_0),x(t_0-\tau),\dots,x(t_0-(m-1)\tau\Big)+(1-\lambda)x(t_0)\bullet x(t_0)\\
               &\geq \lambda x(t_0)\bullet\ff(x(t_0),\dots,x(t_0-(m-1)\tau))\\
               &> 0,
   \end{align*}
   which is the contradiction with the assumption that $\phi(t_0)$ is the maximum of $\phi(t)$, i.e. $\phi''(t_0)\leq 0$.
   In the case $\lambda=0$, the statement is obvious.
   \end{proof}
   
   \begin{lemma}\label{lm32}
Assume that $\ff : \bfV^m\to \bfV$ is continuous and satisfies  \ref{c3}. Then, there exists a constant $M>0$ such that for every solution  $x(t)$ to \eqref{eq3} (for some $\lambda \in [0,1]$),  one has: 
	\[
	\forall_{t\in \br}\;\;\; |x(t)|<M, \; \; |\dot x(t)|<M, \;\; |\ddot x(t)|<M.
	\]
	   \end{lemma} 
    
   \begin{proof}
    By Lemma \ref{lm31}, there exists $R>0$ such that any $2\pi$-periodic solution $x(t)$ to \eqref{eq3} satisfies $|x(t)|\leq R$. 
    
    Put 
    \begin{align*}
     K_R:=\{(\lambda,\bfx)\in[0,1]\times \bfV^m :\;\bfx=(x,y^1,\dots,y^{m-1});\;|x|\leq R,\; |y^j|\leq R,j=1,\dots,m-1\}.
    \end{align*}
Clearly, the set $K_R$ is compact. Since the map $\wt{\ff}(\lambda,\bfx):=\lambda((\ff(\bfx)-x))+x$, $\bfx=(x,y^1,y^2,\dots,y^{m-1})$, $\lambda\in[0,1]$, is continuous,  it follows that  $\wt{\ff}(K_R)$ is bounded, i.e. there exists $M_1\geq 0$ s.t. $|\wt{\ff}(\lambda,\bfx)|\leq M_1$ for all $(\lambda,\bfx)\in K_R$. Therefore, every solution $x(t)$ to \eqref{eq3} satisfies $|\ddot{x}(t)|\leq M_1$.
    
Take  $v\in \bfV$ with $|v|\le 1$ and consider the scalar function $\psi(t):= x(t)\bullet v$.    Since $x(t)$ is  $2\pi$-periodic, there exists $t_0$ such that $\psi'(t_0)=0$ for some $t_0\in \br$, and for $t_o+2\pi\ge t\ge t_0$ one has:
    \begin{align*}
    | \dot x(t)\bullet v|&= |\psi'(t)|=\left|\psi'(t_o)+\int_{t_o}^t \psi''(s)ds\right|=\left|\int_{t_o}^t \psi''(s)ds\right|\\
     &=\left|\int^t_{t_0}\ddot{x}(s)\bullet v ds\right|\le \int^t_{t_0}|\ddot{x}(s)\bullet v | ds  \le \int^t_{t_0}|\ddot{x}(s)|\, | v | ds\\\
     &\le \int^t_{t_0}|\ddot{x}(s)|ds\leq\int^{2\pi}_0|\ddot{x}(s)|ds\leq2\pi M_1=:M_2.
    \end{align*}
    Therefore, 
    \[
    |\dot x(t)|=\sup_{|v|\le 1} |\dot x(t)\bullet v|\le M_2.
    \]
   Summing up,  $M := \max(R,M_1,M_2)+1$ is as required.
   \end{proof}  	
  
  \section{Operator Reformulation in Function Spaces}\label{sec:operator-reform}
   \subsection{Spaces}
  Consider the space $C_{2\pi}(\bbR;\bfV)$ of continuous $2\pi$-periodic functions equipped with the norm
  \begin{equation}
   \|x\|_{\infty}=\sup_{t\in\bbR}|x(t)|,\quad x\in C_{2\pi} (\bbR;\bfV).
  \end{equation}
  Let $\scrE :=C^2_{2\pi}(\bbR,\bfV)$ denote the  space of $C^2$-differentiable  $2\pi$-periodic functions from $\mathbb{R}$ to $\bfV$
  equipped with the norm
  \begin{align}
   \norm{x}_{\infty,2}&=\mbox{max}\{\|x\|_{\infty},\|\dot{x}\|_{\infty},\|\ddot{x}\|_{\infty}\}.
  \end{align}
 Let $O(2)$ denote the group of orthogonal $2\times 2$ matrices. Notice that $O(2)= SO(2)\cup SO(2) \kappa$, where $\kappa=\begin{bmatrix}1&0\\0&-1\end{bmatrix}$, and $SO(2)$ denote the group of rotations $\begin{bmatrix}\cos\tau&-\sin\tau\\\sin\tau&\cos\tau\end{bmatrix} $ which can be identified with $e^{i\tau}\in S^1\subset\mathbb{C}$. Notice that $\kappa e^{i\tau}=e^{-i\tau}\kappa$.\par	
  
  Put $G:=O(2)\times \Gamma\times\mathbb{Z}_2$ and define the $G$-action on $\scrE$ by
    \begin{align}
    (e^{i\theta}, \gamma,\pm 1)x(t)&:=\pm\gamma x(t+\theta),\label{action1}\\
    (e^{i\theta}\kappa, \gamma,\pm 1)x(t)&:=\pm\gamma x(- t+ \theta).\label{action2},
  \end{align}
 where $x\in\scrE,\; e^{i\theta},\kappa \in O(2),\; \gamma \in \Gamma$ and $ \pm 1 \in \mathbb Z_2$. Clearly, $\scrE$ is  an isometric  Banach 
 $G$-representation. 
 In a standard way, one can identify a $2\pi$-periodic function $x:\bbR\rightarrow \bf V$ with a function $\tilde{x}: S^1\rightarrow \bfV$, so one can write $C^2(S^1,\bfV)$ instead of $C^2_{2\pi}(\bbR,\bfV)$. Similar to \eqref{action1}-\eqref{action2} formulas define isometric $G$-representations on the spaces 
of periodic functions $C_{2\pi}(\bbR,\bfV)$ and $L^2_{2\pi}(\mathbb R;V)$ to which appropriate identifications are applied. 
One can easily describe the $G$-isotypic  decomposition of $\scrE $. Consider, first, $\scrE $ as an $O(2)$-representation
corresponding to its Fourier modes: 
\begin{align}\label{dcp}
\scrE=\overline{\bigoplus\limits_{k=0}^{\infty}\mathbb{V}_k},\quad\mathbb{V}_k:=\{\cos(kt)u+ \sin(kt) v:u,\, v\in \bfV\},
\end{align}
where each $\mathbb V_k$, for $k\in \bn$, is equivalent to the complexification $\bfV^c := \bfV \oplus i \bfV$ (as a {\it real} $O(2)$-representation)  of 
$\bf V$, where the rotations $e^{i\theta}\in SO(2)$ act on vectors $\bold z\in \bfV^c$ by $e^{i\theta}(\bold z) :=e^{-ik\theta}\cdot \bold z$ (here `$\cdot$'  stands for complex multiplication) and $\kappa \bold z:=\overline {\bold z}$. Indeed, the linear isomorphism $\vp_k : \bfV^c\to \mathbb V_k$ given by 
\begin{equation}\label{eq:complexification}
\vp_k(x+iy):= \cos(kt) u + \sin(kt) v, \quad u,\, v\in \bfV,
\end{equation}
is $O(2)$-equivariant. Clearly, $\mathbb V_0$ can be identified with $\bfV$ with the trivial $O(2)$-action, while $\mathbb V_k$, $k = 1,2,\ldots$, 
is modeled on the irreducible $O(2)$-representation $\cW_k\simeq\mathbb{R}^2$, where $SO(2)$ acts by $k$-folded rotations and $\kappa$ acts by complex conjugation. 
 
Next, each $\mathbb V_k$, $k = 0,1,2,\ldots$, is also $\Gamma \times \mathbb Z_2$-invariant. Let  $\cV_0^-, \cV_1^-,\cV_2^-,\dots, \cV_{\mathfrak r}^-$
be a complete list of all irreducible orthogonal $\Gamma \times \mathbb Z_2$-representations on which $\Gamma \times \mathbb Z_2$-isotypic  components of $\bfV \simeq \mathbb V_0$ are modeled (here ``$^-$" stands to indicate the antipodal $\mathbb Z_2$-action and $\cV_0^-$ corresponds to the trivial $\Gamma$-action). Since 
$\cV_{k,l}^- :=\cW_k\otimes \cV_l^-$ is an irreducible orthogonal $G$-representation, it follows that  $\mathbb V_0$ and   $\mathbb V_k$ (cf. \eqref{dcp}) admit the following $G$-isotypic  decompositions:
 \begin{equation}\label{iso-0}
     \mathbb V_0=V_{0}^-\oplus V^-_{1}\oplus\dots \oplus V^-_{\mathfrak r}
 \end{equation}
 (with the trivial $O(2)$-action) and 
    \begin{equation}\label{eq:iso-k}
        \mathbb V_k = V_{k,0}^-\oplus V_{k,1}^-\oplus\dots \oplus V_{k,\mathfrak r}^-,
        \end{equation}
where  $V_l^-$ (resp. $V_{k,l}^-$) is modeled on  $\cV_{0,l}^-$  (resp. $\cV_{k,l}^-$ with $k > 0$).    
\begin{remark}\label{rem:non-constant-solutions}        
Clearly, $x \in C^2(S^1;\bfV)$ is not a constant function if $G_x$ does {\it not} contain $O(2) \simeq O(2) \times \{1\} \times \{1\} < G$.
\end{remark}

  \subsection{Operators} 
   Define the following operators:
   \begin{alignat*}{3}
    L:&\scrE \rightarrow C(S^1,\bfV),\quad &Lx&:=\ddot{x}-x\\
    j : &\scrE\rightarrow C(S^1,\bfV^m),\quad & j(x)(t)&:=(x(t),x(t-\tau),\dots,x(t-(m-1)\tau))\\
   N : & \;C(S^1,\bfV^m) \rightarrow  C(S^1,\bfV),\quad 
   &N(x(t),y^1(t),\dots, y^{m-1}(t))&= \ff(x(t),y^1(t),\dots,y^{m-1}(t))-x(t)
   \end{alignat*}
  which can be illustrated on the diagram following below: 
   \begin{figure}[h!]\label{pict}
    \centering
  	\begin{tikzpicture}[>=angle 90]
  	
  	\matrix(a)[matrix of math nodes,
  	row sep=5.5em, column sep=2.5em,
  	text height=1.5ex, text depth=2ex]
  	{\scrE& & C(S^1,\bbV)\\
  		& C(S^1,\bbV^m) \\};
  	\path[->, ]  (a-1-1) edge node[above]{$L$}
  	                                   (a-1-3);
  	\path[->](a-1-1) edge node[below left]{$j$}
  	                      (a-2-2);
  	\path[<-](a-1-3) edge node[below right]{$N$} (a-2-2);
  	
  	\end{tikzpicture}
 
  	\caption{Operators involved}
   \end{figure}
  
  \newpage
  \noindent
  System \eqref{eq3} is equivalent to
   \begin{align}\label{eq4}
    Lx=\lambda \big(N(jx)),\quad  x\in\scrE, \lambda\in[0,1],
   \end{align}
  which is equivalent to \eqref{eq2}  for $\lambda = 1$.  Since $L$ is an isomorphism, equation \eqref{eq4} can be reformulated as follows:
   \begin{align}\label{eq51}
    \scrF_{\lambda}(x):=x-\lambda L^{-1}N(j(x))=0, \;\; x\in\scrE, \;\; \lambda \in [0,1]. 
   \end{align}
   
\begin{proposition}\label{prop}
Suppose that  $f$ satisfies conditions (R), \ref{c1}--\ref{c2} (cf. Remark \ref{rem:same-assumptions}) and  the nonlinear operator $\scrF_{\lambda}:\mathscr E\to \mathscr E$ is given by \eqref{eq51}. Then,  $\scrF_{\lambda}$ is a $G$-equivariant completely continuous field for every $\lambda\in [0,1]$.
   \end{proposition}
\begin{proof}
Combining \eqref{dcp} and  \eqref{eq:complexification} with the definition of $L$ yields:
\begin{equation}\label{eq:action-L}
L|_{\mathbb V_k} =-(k^2+1) \id:V^c\to V^c \quad \rm{and} \quad L|_{\mathbb V_0}=-\id \quad (k > 0). 
\end{equation}
In particular, $L$ (and, therefore, $L^{-1}$) is $G$-equivariant. Since $j$ is the embedding, it is $G$-equivariant as well. Since 
$L$ and $N$ are continuous and $j$ is a compact operator, it follows that  $\scrF_{\lambda}$ is a completely continuous field. 
Also, by assumption \ref{c1} (resp.  \ref{c2}), 
the operator $N$ is $\Gamma$-equivariant (resp. $\bz_2$-equivariant). 
Since system \eqref{weq} is autonomous, it follows that $N \circ j$ is $SO(2)$-equivariant. To complete the proof, one only needs to show that $N \circ j$ commutes with the $\kappa$-action. 
In fact,  combining   the definitions of the actions of $N$, $j$ and $\kappa$ with $2\pi$-periodicity of $x \in \scrE$ and condition $(R)$, one obtains for all
$t$:  
 \begin{align*}
N (j(\kappa x))(t) &=  N (j(x))(-t) =  N\Big(x(-t),x(-t-\tau),\dots,x(-t-(m-1)\tau)\Big) \\  
&= \ff\Big(x(-t),x(-t-\tau),\dots,x(-t-(m-1)\tau)\Big) - x(-t) \\
&\mathop{=}\limits^{(R)} \ff\Big(x(-t), x(-t-(m-1)\tau),  \dots, x(-t-\tau)\Big) - x(-t) \\
&=  \ff\Big(x(-t), x(-t +\tau -2\pi), \dots, x(-t + (m-1)\tau - 2\pi)\Big) -   x(-t) \\
&= \ff\Big(x(-t),x(-t + \tau),\dots,x(-t + (m-1)\tau)\Big) - x (-t)\\
&= \kappa \Big(f\Big(x(t), x(t -\tau), \dots, x(t - (m-1)\tau)\Big) - x (t) \Big)\\
&= \kappa (N (jx))(t)  
\end{align*}
 \end{proof}

 On the other hand,  $x\equiv 0$ is a solution to equation \eqref{eq51} for any $\lambda \in [0,1]$. Assuming that condition \ref{c5} is satisfied, put
 \begin{equation}\label{eq:linearization}
     \scrA:=D\scrF_1(0) : {\scrE} \longrightarrow {\scrE}.
 \end{equation}
 Then,
   \begin{equation}\label{eq:linearization-formula}
    \scrA=\id-L^{-1}(DN(0))\circ j:\scrE\longrightarrow\scrE.
   \end{equation}
 One can easily check that $\scrA$ is a Fredholm operator of index zero.  Therefore, $\scrA$ is an isomorphism if and only if  $0\not\in \sigma(\scrA)$. 
 Also, since $G$ does not move the origin, it follows that  $\scrA$ is $G$-equivariant.

 \begin{lemma}\label{lm33}
 Under the assumptions (R), \ref{c1}, \ref{c2} and \ref{c5}, suppose, in addition, that $0 \not\in \sigma(\scrA)$ (here $\sigma(\scrA)$ stands for the spectrum of $\scrA$).
 Then, for a sufficiently small $\epsilon>0$, 
the map $\scrF := \scrF_1$ (cf.  \eqref{eq51}) is $\Omega_{\varepsilon}$-admissibly  $G$-equivariantly homotopic to $\scrA$ given by 
\eqref{eq:linearization}-\eqref{eq:linearization-formula}
(here  $\Omega_{\varepsilon}:=\big\{x \in \scrE : \|x\|<\epsilon\big\}$).
\end{lemma}
\begin{proof}
Put $\scrH_t(x):=(1-t)\scrA(x)+t\scrF(x), x\in\scrE, t\in[0,1]$, and show that there exists a sufficiently small $\epsilon>0$ such that $\scrH_t(\cdot)$ is 
an $\Omega_{\epsilon}$-admissible homotopy. Indeed, assume for contradiction, that there exist sequences  $\{x_n\} \subset\scrE$ and $\{t_n\}\subset[0,1]$ such that $x_n\rightarrow 0, t_n\rightarrow t_0$ and
\begin{equation*}
\scrH_{t_n}(x_n)=\scrA(x_n)-t_n(\scrA(x_n)-\scrF(x_n)) = 0 \quad\text{for all $n \in \mathbb N$.}
\end{equation*}
Then, by linearity and differentiability, one has:
\begin{equation}\label{eq:A-to-0}
\frac{\scrA(x_n)}{\|x_n\|} = {\scrA\Big(\frac{x_n}{\|x_n\|}\Big) = \frac{t_n(\scrA(x_n)-\scrF(x_n))} 
{\|x_n\|}\rightarrow 0  \quad \text{as}\;\; n\rightarrow \infty}.
\end{equation}
Put $v_n:=\frac{x_n}{\|x\|}$. Combining \eqref{eq:A-to-0} with \eqref{eq:linearization-formula} yields: 
\begin{equation}\label{eq:comp-converge}
\scrA(\itv_n)=\itv_n-L^{-1}(DN(0)(j(\itv_n)))
\rightarrow 0 \;\; \mbox{as}\; n \rightarrow\infty .
\end{equation}
 \noindent 
 Since $j$ is a compact operator, 
 there exist $y_0$ and a subsequence $\{v_{n_k} \}$ such that 
 $L^{-1}(DN(0)(j(\itv_{n_k})))\rightarrow y_0 $. Hence, by continuity of $\scrA$ combined with \eqref{eq:comp-converge}, one has 
 $v_{n_k}\rightarrow y_0$ and $\| y_0 \|=1$.
 Therefore, $\scrA(y_0)=0$ which is impossible since $\scrA$ is an isomorphism.
   	 \end{proof}

\subsection{Abstract equivariant degree based result} 
	
To formulate an equivariant degree based result related to problem \eqref{eq2}, we need additional concepts.
\begin{definition}\label{def:extended-maximal}
  	
	(a) An orbit type $(H)$ in the space $\scrE$ is said to be of \textit{maximal kind} if there exists $k\ge 0$ and $u\neq 0, u\in \bbV_k$, such that $H=G_u$ and $(H)$ is a {\it maximal} orbit type in $\Phi (G,\bbV_k\setminus\{0\})$.
	
	\medskip
	
	(b) Take $x \in \scrE$ and assume that there exists $p \in \mathbb N$ such that $(\phi_p({G}_x)) = (H)$, where $(H)$ is of maximal kind and 
the homomorphism 	$\phi_p : O(2)\times \Gamma\times\bbZ_2\rightarrow O(2)\times \Gamma\times\bbZ_2$ is given by   
\begin{equation*}
\phi_p(g, h,\pm1)=(\mu_p(g), h,\pm1), \quad g\in O(2), \;\; h\in \Gamma
\end{equation*}
(here $\mu_p : O(2)\rightarrow O(2)/ \bbZ_p \simeq O(2)$ is the natural $p$-folding homomorphism of $O(2)$ into itself). Then, $x$ is said to have  an \textit{extended orbit type} $ (H)$.
 \end{definition}
 
 \begin{remark}\label{rem:extended-type}
 The above concepts have a very transparent meaning. Without extra assumptions, typical equivariant degree based results provide 
 {\it minimal spacious symmetries} of the corresponding periodic solutions (cf. Definition \ref{def:extended-maximal}(a)) and do {\it not} provide an information on its {\it minimal  
 period} (Definition \ref{def:extended-maximal}(b)).
 \end{remark} 

\medskip

Under the assumptions (R), \ref{c1}, \ref{c2} and \ref{c5}, the $G$-equivariant degree $\gdeg(\scrA,B(\scrE)) \in A(G)$ is correctly defined
provided that $0 \not\in \sigma(\scrA)$ (here $B(\scrE)$ denotes the unit ball in $\scrE$). Put 
\begin{equation}\label{eq5}
\omega := (G) - \gdeg(\scrA,B(\scrE)).
\end{equation}
We are now in a position to formulate the abstract result.
 \begin{proposition}\label{th:abstract}
 Assume that $f : \bfV^m \to \bfV$ satisfies conditions (R), \ref{c1}\---\ref{c5}. Assume, in addition, that  $0 \not\in \sigma(\scrA)$ (cf.
\eqref{eq51}, \eqref{eq:linearization}, \eqref{eq:linearization-formula}). Assume, finally,
\begin{equation}\label{eq:degree-actual}
       \omega=n_1(H_1)+n_2(H_2)+\dots +n_m(H_m),\quad n_j\neq 0,(H_j)\in\Phi_0(G)
\end{equation}
(cf. \eqref{eq5}). Then:
\begin{itemize}
\item[(a)] for every $j=1,2,\dots,m$, there exists a ${G}$-orbit of $2\pi$-periodic solutions $x\in \scrE\setminus\{0\}$ to  \eqref{eq2} such that 
  $(G_x) \geq (H_j)$; 
\item[(b)] if $H_j$ is finite, then the solution $x$ is non-constant; 
\item[(c)] if $(H_j)$ is of maximal kind, then the solution $x$ has the extended orbit type $(H_j)$.
\end{itemize}
\end{proposition}
\begin{proof}
(a) Take $\Omega_{\epsilon}$ provided by Lemma \ref{lm33}. Then, $\scrF$ is $\Omega_{\epsilon}$-admissible and, by 
equivariant homotopy invariance of the equivariant degree,      
\begin{equation}\label{eq:scrF-degree-zero}
\gdeg(\scrF,\Omega_{\epsilon})=\gdeg(\scrA,B(\scrE)).
 \end{equation}
Similarly, take $M$ provided by Lemma \ref{lm31} and put  $\Omega_M := \big\{x \in \scrE : \|x\| < M \big\}$. Then,
$\scrF = \scrF_1$ is $\Omega_{M}$-admissible and equivariantly homotopic to $\scrF_0 = \id$. Hence,
\begin{equation}\label{eq:scrF-degree-infty}
\gdeg(\scrF,\Omega_{M})=(G).
\end{equation}
Combining \eqref{eq5}, \eqref{eq:degree-actual},  \eqref{eq:scrF-degree-zero}, \eqref{eq:scrF-degree-infty} with the existence property of the equivariant degree yields part (a).

\medskip

(b) Follows from Remark \ref{rem:non-constant-solutions}.

\medskip

(c) Follows from Definition \ref{def:extended-maximal}.   
\end{proof}

\section{Computation of $\gdeg(\scrA,B(\scrE))$}\label{sec:degree-computation}

Proposition \ref{th:abstract} reduces the study of problem \eqref{eq2} to computing $\gdeg(\scrA,B(\scrE))$. In this section,
we will develop a ``workable" formula for $\gdeg(\scrA,B(\scrE))$ and analyze the non-triviality of some of  its coefficients.

\subsection{Spectrum of $\scrA$}
To begin with, we collect the equivariant spectral data related to $\scrA$. Since $\scrA$ is $G$-equivariant, it respects isotypic  decomposition  
\eqref{dcp}. Put $\gamma:=e^{\frac{i2\pi}{m}}$ and $\scrA_k:=\scrA|_{\mathbb V_k}$. Keeping in mind the commensurateness of delays in problem 
\eqref{eq2} and taking into account \eqref{eq:action-L}, one easily obtains:
\begin{equation}\label{eq:Ak}
\scrA_k =\id+\frac{1}{k^2+1}\left(\sum _{j=0}^{m-1}\gamma^{jk}A_j-\id\right), \quad k=0,1,2\dots,
\end{equation}
where $A_j$ stands for the derivative of $\ff$ with respect to $j$-th variable 
(to simplify notations, we keep for derivatives of $\ff$ the same symbols as for the ones of $f$; cf. assumption \ref{c5}). By assumption (R),  
$A_j = A_{m - j}$ for $j = 1,...,m-1$, hence \eqref{eq:Ak} can be simplified as follows: 
\begin{equation}\label{eq:Ak1}
\scrA_k =\id+\frac{1}{k^2+1}\left(A_0+\sum _{j=1}^{r} 2\cos\frac {2\pi jk}{m}A_j-\ve_m A_r-\id\right), \quad k=0,1,2\dots,\; r=\left\lfloor\frac {m-1}{2} \right\rfloor ,
\end{equation}
where 
\begin{equation}\label{eq:epsilon-m}
\ve_m = \begin{cases} 
1\quad  \text{if $m$ is even};\\
 0 \quad  \text{otherwise}.
\end{cases}
\end{equation}
   
 Since   the matrices $A_j$ are $\Gamma$-equivariant, one has $\scrA_k(V_{k,l}^-)\subset V_{k,l}^-$ $(k = 0,1,2,\ldots$ and   $l=0,1,2,\dots, \mathfrak r)$.
  In particular, 
   $A_j(V_l^-)\subset V_l^-$, so put
   \[
   A_{j,l}:=A_j|_{V_l^-}, \quad l=0,1,2,\dots, \mathfrak r.
   \] 
To simplify the computations, we will assume 
that instead of \ref{c5} the following condition is satisfied:
 \begin{itemize}
\item [($A_4'$)] $A_{j,l}=\mu_j^l \id$ for $l=0,1,2,\dots, \mathfrak r$ and $j=0,1,\dots,m-1$.
   \end{itemize}
Clearly, under condition ($A_4'$),  the matrices $A_j$  commute with each other, therefore, condition \ref{c5} follows. 
In particular, their corresponding  eigenspaces coincide: $E(\mu_j^l)=E(\mu_{j'}^l)$. This way, one 
 obtains the following description of the spectrum of $\scrA$:
 \begin{equation}\label{eq:spectru} 
 \sigma(\mathscr A)=\bigcup_{k=0}^\infty \sigma(\mathscr A_k),
 \end{equation}
 where 
   \begin{equation}\label{spcta}
    \sigma (\scrA_k)=\left\{1+\frac{1}{1+k^2}\left(\mu^l_0+\sum _{j=1}^{r}2\cos\frac {2\pi jk}m\mu^l_j-\ve_m\mu_r^l-1\right): l=0,1,\dots,\mathfrak r, \; 
     r=\left\lfloor\frac {m-1}{2} \right\rfloor\right\}.
   \end{equation}
   
\bigskip

\subsection{Computation of $\gdeg(\scrA,B(\scrE))$: reduction to basic $G$-degrees}
Observe that $\xi_{k,l} \in \sigma (\scrA_k)$ contributes $\gdeg(\scrA,B(\scrE))$ only if $\xi_{k,l} < 0$. Clearly (cf. \eqref{spcta}),
\begin{equation}\label{eq:eigen-kl}
\xi_{k,l}:=1+\frac{1}{1+k^2}\left(\mu_0^l +\sum _{j=1}^{r} 2\cos \frac{2\pi jk}{m}\mu^l_j-\ve_m \mu_r^l-1 \right), \;  l=0,1,\dots,\mathfrak r, \; r=\left\lfloor\frac {m-1}{2} \right\rfloor , \; k = 0,1,...
\end{equation}
is negative (i.e. $\xi_{k,l}\in \sigma_-(\scrA)$) if and only if
\begin{align}\label{ieqk}
k^2 <-\mu_0^l-\sum _{j=1}^r  2\cos \frac{2\pi jk}{m}\mu_j^l+\ve_m\mu_r^l, \quad  l=0,1,\dots,\mathfrak r, \; r=\left\lfloor\frac {m-1}{2} \right\rfloor , \; k = 0,1,...
\end{align}
By condition ($A_4'$), the $\cV_l^-$-isotypic  multiplicity of $\mu_j^l$ is independent of $j$ and is equal to 
\begin{equation}\label{eq:m-l}
m^l := \text{dim\,} E(\mu_j^l)/\text{dim\,} \cV_l^-=\text{dim\,}V_{l}^- /\text{dim\,} \cV_l^-.
\end{equation}
Put (cf. \eqref{ieqk}-\eqref{eq:m-l})
\begin{align}\label{mtpl}
m_{k,l}:=\begin{cases} 
m^l &\; \text {if } \;   k^2 <-\mu_0^l-\sum _{j=1}^{r}  2\cos \frac{2\pi jk}{m}\mu_j^l+\ve_m\mu_r^l  \\
0 & \text{ otherwise}.
\end{cases}
\end{align}
Then,   
\begin{align}\label{eq6}
\gdeg(\scrA,B({\scrE}))&=\prod_k 
\prod_{l=0}^{\mathfrak r}\big(\mbox{deg}_{\cV_{k,l}^-}\big)^{m_{k,l}}
\end{align}
Notice that in the product \eqref{eq6},  one has $m_{k,l}\not=0$ for finitely many values of $k$ and $l$ (cf. \eqref{mtpl}). Hence,  
for almost all the factors in \eqref{eq6}, one has $(\mbox{deg}_{\cV_{k,l}})^0=({G})$, which is the unit element in $A({G})$. Thus,  formula   \eqref{eq6}  is well-defined. 
 
\begin{remark}\label{rem:odd-powers-only}   
Using the relation $(\mbox{deg}_{\cV_{k,l}^-})^2=(G)$, one can further simplify formula \eqref{eq6}.  Clearly, only the exponents $m_{k,l}\not=0$ which are odd will contribute to the value of \eqref{eq6}. 
\end{remark}

\subsection{Maximal orbit types in  products of basic $G$-degrees}
In order to effectively apply Proposition \ref{th:abstract}(c), one should answer the following question: which orbit types of maximal kind (see Definition \ref{def:extended-maximal}) appearing
in the right-hand side of formula \eqref{eq6} will ``survive" in the resulting product? 

To begin with, take $\deg_{\cV_{k,l}^-}$ appearing in \eqref{eq6} and let $(H_o)$ be a maximal  orbit type in $\cV_{k,l}^-\setminus \{0\}$. Then,  
\begin{equation}\label{eq:basic}
\deg_{\cV_{k,l}^-}=(G)- x_o(H_o)+a,\quad     -x_o:=\frac{(-1)^{\text{dim}\cV_{k,l}^{-H_o}}-1}{|W(H_o)|},
\end{equation}
where $a \in A(G)$ has  a zero coefficient corresponding to $(H_o)$.  
Then, by  \eqref{eq:basic}, one has  
\begin{equation}\label{eq:coef-x-o-ireduc}
x_o=\begin{cases}
0 & \; \text{if $\text{dim}\cV_{k,l}^{-H_o}$ is even}\\ 
1 & \text{ if $\text{dim}\cV_{k,l}^{-H_o}$ is odd  and $|W(H_o)|=2$}\\
2 &    \text{ if $\text{dim}\cV_{k,l}^{-H_o}$ is odd  and $|W(H_o)|=1$}.
\end{cases}
\end{equation}
  
Notice that for two different irreducible $G$-representations $\cV_{k,l}^-$ and $\cV_{k',l'}^-$, it is possible that $\deg{\cV_{k,l}^-}=\deg{\cV_{k',l'}^-}$ 
(see, for example, \cite{AED}, p. 183, where irreducible $\Gamma = A_5$-representations are discussed). In this case, one has 
$\deg{\cV_{k,l}^-} \cdot \deg{\cV_{k',l'}^-} = (G)$ (cf. Remark \ref{rem:odd-powers-only}). 

\begin{lemma}\label{lem:same-maximal}
Suppose that $(H_o)$ is a maximal orbit type in $\cV_{k,l}^-\setminus\{0\}$ and $\cV_{k',l'}^-\setminus \{0\}$ and both $\cV_{k,l}^{-H_o}$ and 
$\cV_{k',l'}^{-H_o}$ are of odd dimension. Then:
\begin{itemize}
\item[(i)] $\deg{\cV_{k,l}^-}$ and $\deg{\cV_{k',l'}^-}$ share the same coefficient  $x_o$ standing 
by $(H_o)$;
\item[(ii)] the coefficient standing by $(H_o)$ in  $  \deg{\cV_{k,l}^-}\cdot \deg{\cV_{k',l'}^-}$ is zero. 
\end{itemize}
\end{lemma}
\begin{proof}
\

(i) Follows immediately from \eqref{eq:coef-x-o-ireduc}.
\smallskip

(ii) Consider the following product
\begin{align*}
\deg{\cV_{k,l}^-}\cdot \deg{\cV_{k',l'}^-}&=  \big( (G)-x_o(H_o)+ a   \big)\cdot \big( (G)-x_o(H_o)+ a'    \big)\\
&= (G)-2x_o(H_o)+x_o^2(H_o)\cdot (H_o)+b,
\end{align*}
where   $a$, $a'$, $b\in A(G)$ are elements with zero coefficient standing by $(H_o)$. 
Then, by using the Recurrence Formula for multiplication, one obtains:
  \[
  (H_o)\cdot(H_o)=y_o(H_o)+ d, \quad y_o:=\frac{n(H_o,H_o)^2|W(H_o)|^2}{|W(H_o)|}= |(W(H_o)|.
  \]
 Hence,
  \[
  -2x_o(H_o)+x_o^2(H_o)\cdot (H_o)=\begin{cases}
 -2+2 & \text{if } x_o=1, \;\; |W(H_o)|=2\\
  -4+4  & \text{if } x_o=2, \;\; |W(H_o)|=1\\
  \end{cases}\cdot (H_o)
  =0,
  \]
and the statement follows.
\end{proof}

\medskip

Observe that the set $\displaystyle \bigcup_{k>0} \Phi_0(G,\mathbb V_k)\setminus\{(G)\}$ has no maximal elements. Indeed,  for an orbit type $(H)$ in $\Phi_0(G,\mathbb V_k)\setminus \{(G)\}$, any orbit type $(H')\in \Phi_0(G,\mathbb V_{pk})\setminus \{(G)\}$ such that $\phi_p(H')=(H)$, satisfies  $(H')>(H)$ (see Definition \ref{def:extended-maximal}). Notice that  if $(H_o)$ is maximal in 
$\Phi_0(G,\mathbb V_k)\setminus \{(G)\}$, then $(H_o')>(H_o)$ for $(H'_o)\in \Phi_0(G,\mathscr E)\setminus \{(G)\}$ if and only if  $(H_o')$ is maximal in  $\Phi_0(G,\mathbb V_{pk})\setminus \{(G)\}$ for some $p > 1$ and $\phi_p(H_o')=(H_o)$. To be more precise, in this a case, if $(H_o)$ is a maximal orbit type in $\cV_{k,l}^-\setminus\{0\}$, then $(H_o')$ is a maximal orbit type in $\cV_{pk,l}^- \setminus \{0\}$. 
Moreover, the statement following below shows that if a maximal type is presented in  $\deg{\cV_{k,l}^-}$, then it ``survives" after the multiplication by 
a ``folded" basic degree.  

\begin{lemma}\label{lem:max-folding}
Suppose that $(H_o)$ is a maximal orbit type in $\cV_{k,l}^-\setminus\{0\}$ and  
let $(H_o')$ be the maximal type in $\cV_{pk,l'}^-\setminus \{0\}$ 
such that $\phi_p(H_o') = (H_o)$ (see Definition \ref{def:extended-maximal}). Then:
\begin{itemize}
\item[(i)] $\deg{\cV_{k,l}^-}$ and $\deg{\cV_{pk,l'}^-}$ share the same coefficient standing 
by $(H_o)$ and $(H_o')$, respectively;
\item[(ii)] if, in addition, the dimension of $\cV_{k,l}^{-H_o}$ is odd (cf. Lemma \ref{lem:same-maximal}), and $x_o$ is the coefficient standing by 
$(H_o)$ in  $  \deg{\cV_{k,l}^-}$, then  the coefficient standing by  $(H_o)$ in $  \deg{\cV_{k,l}^-}\cdot \deg{\cV_{k',l'}^-}$ is equal to $-x_o$ (in particular,
different from zero). 
\end{itemize}
\end{lemma}
\begin{proof} (i) Take  
\begin{equation}\label{eq:deg-prod-folde}
\deg{\cV_{k,l}^-} =   (G) - x_o(H_o) + a \quad \text{and} \quad \deg{\cV_{pk,l'}^-} = (G) - x_o'(H_o')+ a', 
\end{equation}
where $a,a' \in A(G)$ have zero coefficients standing by $(H_o)$ and $(H_o')$, respectively. Since $|W(H_o)|=|W(H_o')|$, the statement follows from  \eqref{eq:coef-x-o-ireduc}.

\smallskip

(ii)
Under the notations of (i), one has
\begin{align*}
\deg{\cV_{k,l}^-}\cdot \deg{\cV_{pk,l'}^-}
&= (G) - x_o(H_o) - x_o(H_o') + x_o^2(H_o)\cdot (H_o')+b,
\end{align*}
where $b\in A(G)$ is the element with zero coefficients standing by $(H_o)$ and $(H_o')$. Then, again by Recurrence Formula, one obtains: 
\[
(H'_o)\cdot(H_o)=y_o(H_o)+ d, \quad y_o=\frac{n(H'_o,H_o)n(H_o,H_o)|W(H'_o)|\, |W(H_o)|}{|W(H_o)|}= |(W(H_o)|,
\]
where $d \in A(G)$ is an element with zero coefficients standing by $(H_o)$ and $(H_o')$. Hence, 
  \[
  -x_o(H_o)+x_o^2(H_o)\cdot (H_o')=\begin{cases}
 -1+2 & \text{if } x_o=1, \;\; |W(H_o)|=2\\
  -2+4  & \text{if } x_o=2, \;\; |W(H_o)|=1\\
  \end{cases}\cdot (H_o)=x_o (H_o).
  \]
Consequently, 
\[
\deg{\cV_{k,l}^-}\cdot \deg{\cV_{pk,l'}^-} = (G) - x_o(H_o') + x_o(H_o)+b.
\]
 \end{proof}
 We summarize the above discussion in the statement following below. To this end, we need additional notations. 
\begin{definition}\label{def:coef-H0} 
   (i) For any $(H_o)\in \Phi_0(G)$, define the function $\text{coeff}^{H_o} : A(G)\to \bz$ assigning  to any  $a=\sum_{(H)} n_H(H) \in A(G)$  the coefficient  $n_{H_o}$  standing by $(H_o)$. 

\smallskip

(ii) Given 
an orbit  type $(H_o) \in \Phi_0(G,\mathscr E)$ of maximal kind (see Definition \ref{def:extended-maximal}(a)) and $k = 0,1,2,\dots,$ define the integer 
\begin{equation}\label{eq:parity}
\mathfrak n^{H_o}_k:=\sum_{l=0}^{\mathfrak r} \mathfrak l^{H_o}_{k,l} \cdot m_{k,l}, 
\end{equation}
where  $m_{k,l}$ is given by \eqref{mtpl}   and
\begin{equation}
\mathfrak l^{H_o}_{k,l}  :=\begin{cases}
1 &\text{ if } \;  \text{dim}\cV_{k,l}^{-H_o} \;\; \text{is odd}\\ 
0 &\text{ otherwise}
\end{cases} 
\end{equation}
(cf. formulas \eqref{eq6}--\eqref{eq:coef-x-o-ireduc}).
\end{definition}    
The following statement is an immediate consequence of Remark \ref{rem:odd-powers-only} and Lemmas \ref{lem:same-maximal} and \ref{lem:max-folding}.     
\begin{lemma}\label{lem:comp-1} Let $(H_o) \in \Phi_0(G,\mathscr E)$ be an orbit type of
maximal kind  (see Definition \ref{def:extended-maximal}(a)) and assume that for some $k\ge 0$, the number $\mathfrak n^{H_o}_k$ is odd 
(see Definition \ref{def:coef-H0}). Then,
\begin{equation}\label{eq:formula1}
\text{\rm coeff\,}^{H_o}\Big(   \gdeg(\scrA,B({\scrE}))  \Big)=\pm x_o,
\end{equation}
where $x_o$ is given by 
\eqref{eq:coef-x-o-ireduc}.
\end{lemma}
      
\section{Main Results and Example}\label{sec:main}\label{sec:main-results-example}
 \subsection{Main result: non-degenerate version}
In this section, we will present our main results and describe a class of illustrating examples with $G = O(2) \times D_6 \times \mathbb Z_2$.      The ``non-degenerate'' version of the main result is:
\begin{theorem}\label{th:main1}
Assume that $f : \bfV^m \to \bfV$ satisfies conditions (R), \ref{c1}\---\ref{c3} and ($A_4'$). Assume, in addition, that  $0 \not\in \sigma(\scrA)$,
where  $\sigma(\scrA)$ is given by \eqref{eq:spectru}--\eqref{spcta} (see also \eqref{eq:epsilon-m}). Assume, finally, that there  
exist  $k\in \bn$ and an orbit type $(H_o)$ in $\Phi_0(G,\mathscr E)$ of maximal kind such that  
$\mathfrak n^{H_o}_k$  is odd (see Definitions \ref{def:extended-maximal}(a) and \ref{def:coef-H0}). 

Then, system  \eqref{eq2} admits a non-constant $2\pi$-periodic solution with the extended orbit type $(H_o)$.
\end{theorem}
\begin{proof}
The proof follows immediately from Lemma \ref{lem:comp-1}  and Proposition \ref{th:abstract}(c).
\end{proof}

\subsection{Example} 
\label{sec:example}
\
{\bf (a) Growth at infinity.} We start with describing a class of maps satisfying condition $(A_3)$.
Take $\bfV:=\br^n$  equipped with the norm {\it max}, and consider 
a map $F : \bfV\times \bfV^{m-1}\to \bfV$ given by
\begin{equation}\label{eq:map-F}
F(x,\bold y)=(p_1(x,\bold y),p_2(x,\bold y),...,p_n(x,\bold y))^T\in \bfV  \quad\quad  (x \in V, \; \bold y \in \bfV^{m-1}),
\end{equation}
where 
\begin{equation}\label{eq:map-p}
p_s(x,\bold y)=x_s^rQ_s(x,\bold y)+q_s(x,\bold y) \quad\quad (s = 1,...,n) ,
\end{equation}
$r\ge 5$ is an odd integer, $Q_s(x,\bold y)$ is a polynomial which is positive 
for all $x$ and $\bold  y$, 
and  $q_s(x,\bold y)$ is a homogeneous polynomial of degree $3 \le  t < r$ (we assume that $t$ is also odd). Choose (for a moment, arbitrary) matrices 
$A_j:\bfV \to\bfV$, $j=0,1, \dots, m-1$, and define 
\begin{equation}\label{eq:map-f}
f(x,\bold y):=A_0x+\sum_{j=1}^{m-1} A_j y^j +F(x,\bold y) \quad\quad (x\in \bfV, \; \bold y\in \bfV^{m-1}).
\end{equation}
Put   $q_o=\min\{Q_s( x,\bold y): s = 1,\dots,n\}$. Since $Q_s$ is a polynomial, $q_o>0$, thus, for $|\bold y|\le |x|$,  one has: 
\begin{align*}
x\bullet f(x,\bold y)&=x\bullet A_0 x+x\bullet \sum_{j=1}^{m-1} A_j y^j +\sum_{s=1}^n\Big({x_s^{r+1}Q_s(x,\bold y)})+x_sq_s(x,\bold y)\Big)\\
&\ge\sum_{s=1}^nx_s^{r+1}q_o-|A_0||x|^2-\sum_{j=1}^{m-1} |A_j|\,| y^j||x|-\sum_{s=1}^n|x_sq_s(x,\bold y)|\\
&\ge q_o|x|^{r+1}-\left(|A_0|+\sum_{j=1}^{m-1} |A_j|\right)\,|x|^2-C|x|^{t+1}-D\\ 
\end{align*}
Since  $  r\ge 5$ and $ t < r$, it follows that  
there exists a constant $R>0$ such that if $|x|>R$, then 
\begin{align*}
x\bullet f(x,\bold y)&\ge q_o|x|^{r+1}-\left(|A_0|+\sum_{j=1}^{m-1} |A_j|\right)\,|x|^2-C|x|^{t+1}-D>0,
\end{align*}
which implies that the assumption ($A_3$) is satisfied.

\medskip

{\bf (b) Symmetries.}  Assume, further, that $\bfV=\br^n$ is an orthogonal $\Gamma$-representation, where $\Gamma\le S_n$ acts on the vectors in $\br^n$ by permuting their coordinates. Then, one can easily identify the $\Gamma$-symmetric interactions between the coordinates in $\bfV$. Following these interactions, one can easily chose monomials for $Q_s$ and $q_s$, $s = 1,...,n$, in such a way that  the map $f$ given by \eqref{eq:map-f} satisfies condition ($A_1$). 

Assume, in addition, that the $\Gamma$-isotypic  decomposition of $\bfV$ is given by
\[
\bfV = V_0 \oplus \dots\oplus V_{\mathfrak r},
\]
where $V_l$, $l=0,1,2,\dots,\mathfrak r$, is equivalent to an irreducible $\Gamma$-representation $\cV_l$ of a real type. Then clearly, for each $j=0,1,2,\dots,m-1$ and $l=0,1,2,\dots,\mathfrak r$, there exists $\mu_j^l\in \br$ such that
\[
A_j|_{V _l}=\mu_j^l\, \id_{V_l}
\]
(just consider any $A_j$ in $V$ in an ``isotypic" basis). Then, since $t > 2$, it follows that the map $f$ also satisfies ($A_4'$). 

Observe that condition ($A_2$) is nicely compatible with other assumptions on $f$ listed in item (a) -- it is enough to assume, in addition, that $Q_s$ is an even polynomial for all $s$. Also, condition $(R)$ can be easily satisfied. 

\medskip

{\bf (c)  $O(2)\times D_6\times \mathbb Z_2$-equivariant problems}.  Let $\Gamma :=D_6$ be the dihedral group given by
\[ 
D_6:=\{1,\gamma,\gamma^2,\gamma^3,\gamma^4,\gamma^5,\kappa,\gamma\kappa,\gamma^2\kappa,\gamma^3\kappa,\gamma^4\kappa,\gamma^5\kappa\},\]
 with  $\gamma:=e^{\frac{\pi i}3}$ and $\kappa z=\overline z$, $z\in \mathbb C$.  Let $\bfV := \bbR^6$ and 
assume that $\Gamma$ acts on $\bfV$ by permuting the coordinates of the vectors $x= (x_1,x_2,\dots, x_6)$ the same way as it permutes vertices of a regular hexagon on the plane. Take $ \bfV^6=\bfV\mathop{\underbrace{\times\cdots\times}\limits_{6\times}}\bfV$ and assume that $\Gamma$ acts diagonally on it. The list of irreducible $D_6$-representation is given by the 
table of their characters:
   \begin{table}[H]
    	\center
    	\begin{tabular}{lllllll}	
    		\toprule
    		&$(1)$&$(\kappa)$&$(\gamma)$&$(\gamma^2)$&$(\kappa \gamma)$&$(\gamma^3)$\\
    		\hline
    		$\chi_1$ & 1& 1& 1& 1& 1& 1\\
    		$\chi_2$ & 1& -1& -1& 1& 1& -1 \\
    		$\chi_3$ & 1& -1& 1& 1& -1& 1\\
    		$\chi_4$ & 1& 1& -1& 1& -1& -1\\
    		$\chi_5$ & 2& 0& 1& -1& 0& -2\\
    		$\chi_6$ & 2& 0& -1& -1& 0& 2\\
    		\bottomrule
    	\end{tabular}	
    	\caption{Character Table of $D_6$}
    \end{table}	
 \noindent
 Notice that the character  $\chi$ of the representation $\bfV$ is: 
 
 \smallskip
   \begin{center}
    \begin{tabular}{lllllll}	
    		\toprule
    		&$(1)$&$(\kappa)$&$(\gamma)$&$(\gamma^2)$&$(\kappa \gamma)$&$(\gamma^3)$\\
    		\hline
    		$\chi$ & 6& 2& 0& 0& 0& 0\\
    		\bottomrule
    	\end{tabular}	
\end{center} 
\smallskip
\noindent
which implies the following $D_6 \times \mathbb Z_2$-isotypic  decomposition of $\bf V$:
\begin{equation}\label{eq:isotyp-decomp-D6}
\bfV:=V_1^-\oplus V_4^-\oplus V_5^- \oplus V_6^-.
\end{equation}

Let $f: \bfV^6\to \bfV$  be given by \eqref{eq:map-f} and satisfy conditions \ref{c1}\---\ref{c3}
(i.e. in this case $m=6$ and $n=6$). Assume that 
\begin{equation}\label{eq:Df-0}
Df(0)=[d\hat{A},a\hat{A},b\hat{A},c\hat{A},b\hat{A},a\hat{A}],
\end{equation}
where 
    \begin{align}
     \hat{A}=\begin{bmatrix}\label{eq:marix-hat-A}
    -1&1/10&0&0&0&1/10\\
    1/10&-1&1/10&0&0&0\\
    0&1/10&-1&1/10&0&0\\
    0&0&1/10&-1&1/10&0\\
    0&0&0&1/10&-1&1/10\\
    1/10&0&0&0&1/10&-1
    \end{bmatrix} ,
\end{align}
and the numbers $a$, $b$, $c$ and $d$ will be specified later.  One can easily verify that  $\hat{A}$ is $\Gamma$-equivariant and 
\begin{equation}\label{eq:spectr-hatA}
\sigma(\hat{A})= \{_1\mu=-8/10, \; _2\mu=-9/10,\;  _3\mu=-11/10,\;  _4\mu=-12/10\}.
\end{equation}  
Since each isotypic  component in \eqref{eq:isotyp-decomp-D6} is irreducible, it follows that condition ($A_4'$) is automatically satisfied,
and one can verify that: 
   \begin{align*}
  \mu_j^1&:=\begin{cases}
  d\cdot {_1\mu}&\text{ if } j=0\\
    a\cdot {_1\mu}&\text{ if } j=1,5\\
      b\cdot {_1\mu}&\text{ if } j=2,4\\
        c\cdot {_1\mu}&\text{ if } j=3
  \end{cases} \quad
  \mu_j^4:=\begin{cases}
  d\cdot {_3\mu}&\text{ if } j=0\\
    a\cdot {_3\mu}&\text{ if } j=1,5\\
      b\cdot {_3\mu}&\text{ if } j=2,4\\
        c\cdot {_3\mu}&\text{ if } j=3
  \end{cases}
\\
  \mu_j^5&:=\begin{cases}
  d\cdot {_2\mu}&\text{ if } j=0\\
    a\cdot {_2\mu}&\text{ if } j=1,5\\
      b\cdot {_2\mu}&\text{ if } j=2,4\\
        c\cdot {_2\mu}&\text{ if } j=3
  \end{cases}
  \quad
  \mu_j^6:=\begin{cases}
  d\cdot {_4\mu}&\text{ if } j=0\\
    a\cdot {_4\mu}&\text{ if } j=1,5\\
      b\cdot {_4\mu}&\text{ if } j=2,4\\
        c\cdot {_4\mu}&\text{ if } j=3
  \end{cases}
   \end{align*}
(cf. \eqref{eq:spectr-hatA}).   
Also (cf. \eqref{eq:m-l}),  $m^l = 1$ for all $l=1,4,5,6$. 
In addition (cf. \eqref{eq:Ak}-\eqref{spcta}), =
one has:
    \[
    d+a\gamma^k+b\gamma^{2k}+c\gamma^{3k}+b\gamma^{4k}+a\gamma^{5k}=\begin{cases}
    d+2a+2b+c& k=0\text{ mod }6\\
    d+a-b-c&k=1\text{ mod }6\\
    d-a-b-c&k=2\text{ mod }6\\
    d-2a+2b-c&k=3\text{ mod }6\\
    d-a-b+c&k=4\text{ mod }6\\
    d+a-b-c&k=5\text{ mod }6
    \end{cases}
    \]
Take 
\begin{equation}\label{eq:a-b-c-d}
a:=4, \quad b:=1,\quad c:=3, d:=6.9.
\end{equation}
Then (cf. \eqref{eq:eigen-kl}-\eqref{ieqk}),  one has the following table:
 \begin{table}[H]
  	\center
  	\begin{tabular}{lllll}	
  		\toprule
  		\multicolumn{5}{c}{$\text{sign\,}\xi_{k,l}$}\\	
  		\hline 
  		$k\backslash l$&1&4&5&6\\
  		\hline
  		0&$-$&$-$&$-$&$-$\\
  		1&$ -$& $-$& $-$& $-$ \\
  		2&$+$&$-$&$-$&$-$\\
  		3&$+$&$+$&$+$&$+$\\
  		\bottomrule
  	\end{tabular}	
  	\caption{Eigenvalues of $\scrA$}
  \end{table}	
\noindent 
Notice that there is no non-negative eigenvalue for $k>2$. Also, $0 \not\in \scrA$. Therefore (cf. \eqref{eq6}),

  \begin{align}
    \gdeg(\scrA,B({\scrE}))& = \mbdeg_{\cV_{0,1}^-}\cdot   \mbdeg_{\cV_{0,4}^-}  \cdot   \mbdeg_{\cV_{0,5}^-}   \cdot   \mbdeg_{\cV_{0,6}^-}           
    \cdot   \mbdeg_{\cV_{1,1}^-}   \cdot  \mbdeg_{\cV_{1,4}^-} \cdot  \mbdeg_{\cV_{1,5}^-} \cdot \mbdeg_{\cV_{1,6}^-}\notag\\
    &\cdot \mbdeg_{\cV_{2,4}^-} \cdot \mbdeg_{\cV_{2,5}^-} \cdot \mbdeg_{\cV_{2,6}^-}. 
    \label{eq:deg-final}
  \end{align}
The orbit types of maximal kind occurring in $\cV_{k,l}^-$ 
with $k > 0$ and $l = 1,4,5,6$ related to \eqref{eq:deg-final} are: \newpage
\begin{equation}\left.
\begin{aligned}
\cV_{1,1}^- \,& : \;\;\;\;      (\amal{D_2}{D_1} {\bz_2} {D_6^z}{D_6^p});\\
\cV_{1,4}^- \, &: \;\;\;\;   (\amal{D_2}{D_1} {\bz_2}{ D_3^p}{D_6^p});\\
\cV_{1,5}^- \, &: \;\;\;\;  (\amal{D_6}{\bz_1} {D_6}{\bz_2^-}{D_6^p}), \; (\amal{D_2}{D_1} {\bz_2} {D_2^d}{D_2^p}), \; 
(\amal{D_2}{D_1}{\bz_2}{\wt D_2^d}{D_2^p});\\
\cV_{1,6}^- \, &: \;\;\;\; (\amal{D_6}{\bz_1} {D_6} {\wt D_1}{D_6^p}),  \; (\amal{D_2}{D_1}{\bz_2}{ D_2^z}{D_2^p}), \;\; (\amal{D_2}{D_1}{\bz_2}{ D_2}{D_2^p});\\
\cV_{2,4}^- \, &: \;\;\;\;  (\amal{D_4}{D_2} {\mathbb Z_2} {D_3^p}{D_6^p});\\
\cV_{2,5}^- \, &: \;\;\;\; (\amal{D_{12}}{\bz_2} {D_6}{\bz_2^-}{D_6^p}), \;  (\amal{D_4}{D_2} {\mathbb Z_2} {D_2^d}{D_2^p}), \; (\amal{D_4}{D_2 }{\bz_2} {\wt D_2^d}{D_2^p});\\
\cV_{2,6}^- \, &: \;\;\;\;   (\amal{D_{12}}{\bz_2} {D_6} {\wt D_1}{D_6^p}), \; (\amal{D_4}{D_2} {\mathbb Z_2} {D_2^z}{D_2^p}), \;  (\amal{D_4}{D_2} {\mathbb Z_2} {D_2}{D_2^p})
\end{aligned}
\right\}
\label{eq:maxOrb}
\end{equation}
\medskip
\begin{remark}\label{rem:notations} (i)  For any subgroup $S  \leq D_6$, the symbol $S^p$ stands for $S \times \mathbb Z_2$.  

\smallskip
(ii) Given two subgroups $H \leq O(2)$ and $K \leq D_6^p$, we refer to Appendix, item (a), for the ``amalgamated notation"  $\amal{H}{Z} {L} {R}{K}$. 

\smallskip
(iii) We refer to  \cite{AED} for the explicit description of the (sub)groups $\wt D_k$, $D_k^z$, $D_k^d$, $\wt D_k^d$, and $\mathbb Z_2^-$.
\end{remark}

We summarize our considerations in the statement following below.
\begin{proposition}\label{prop:example-summarize}
Let $\bf V$ be a $D_6 \times \mathbb Z_2$-representation admitting isotypic decomposition \eqref{eq:isotyp-decomp-D6} and let 
$\bfV^6$ be equipped with the diagonal $D_6 \times \mathbb Z_2$-action.  
Let $f: \bfV^6\to \bfV$  be given by \eqref{eq:map-F}-\eqref{eq:map-f} with $Q_s$ being an even polynomial ($s = 1,...,n$), and suppose that 
conditions \text{(R)} and \ref{c1} are satisfied.
Assume, further, that $Df(0)$ is given by \eqref{eq:Df-0}-\eqref{eq:marix-hat-A} and \eqref{eq:a-b-c-d}.  Let $(\mathcal H)$ be one of the orbit types listed
in \eqref{eq:maxOrb}. Then,  system  \eqref{eq2} admits a non-constant $2\pi$-periodic solution with the extended orbit type $(\mathcal H)$.

\end{proposition}

\subsection{Main result: degenerate version}
Using the argument similar to the one utilized in the proof of Theorem \ref{th:main1}, one can establish the following  ``degenerate'' version of the main result. 

\begin{theorem}\label{th:main2}
Assume that $f : \bfV^m \to \bfV$ satisfies conditions (R), \ref{c1}\---\ref{c3} and ($A_4'$). 
Put 
\[
\mathscr C:= \left\{ k\in \bn\cup\{0\}:
k^2  = -\mu_0^l-\sum _{j=1}^r  2\cos \frac{2\pi jk}{m}\mu_j^l+\ve_m\mu_r^l, \;\; l=0,1,2,\dots, \mathfrak r, \;  r:=\left\lfloor\frac {m-1}{2} \right\rfloor \right\}
\] 
and choose 
$s \in \bn$ such that 
\begin{equation}\label{eq:non-degenerate1}
\mathscr C\cap \{ (2k-1)s: k\in \bn\}=\emptyset.
\end{equation}
Assume that there  
exist  $k\in \bn$ and an orbit type $(H_o)$ in $\Phi_0(G,\mathscr E)$ of maximal kind such that  
$\mathfrak n^{H_o}_{(2k-1)s}$  is odd (see Definitions \ref{def:extended-maximal}(a) and \ref{def:coef-H0}).
Then, system  \eqref{eq2} admits a non-constant $2\pi$-periodic solution with the extended orbit type $(H_o)$. 
\end{theorem}

\begin{remark}\label{rem:deg-theorem}
It is easy to extend the setting considered in Proposition \ref{prop:example-summarize} to the one supporting Theorem \ref{th:main2}. We leave this task to the reader. 
\end{remark}

\appendix
\section{Equivariant Brouwer Degree Background}
\label{subsec:G-degree}

\
{\bf (a) Amalgamated notation.} 
Given two groups $G_{1}$ and
$G_{2}$, 
the well-known result of \'E. Goursat (see \cite{DKLP,Goursat}) provides the following description of a
subgroup $\mathscr H \leq G_{1}\times G_{2}$: 
there exist subgroups
$H\leq G_{1}$ and $K\leq G_{2}$, a group $L$, and two epimorphisms
$\varphi:H\rightarrow L$ and $\psi:K\rightarrow L$ such that
\begin{equation*}
\mathscr H=\{(h,k)\in H\times K:\varphi(h)=\psi(k)\}.
\end{equation*}
The widely used notation for $\mathscr H$ is 
\begin{equation}\label{eq:amalgam-projections}
\mathscr H:=H\prescript{\varphi}{}\times_{L}^{\psi}K,
\end{equation}
in which case $H\prescript{\varphi}{}\times_{L}^{\psi}K$ is called an
\textit{amalgamated} subgroup of $G_{1}\times G_{2}$.

In this paper, we are interested in describing conjugacy classes of $\mathscr H$. Therefore, to make notation \eqref{eq:amalgam-projections} simpler and
self-contained, it is enough to indicate $L$,  
$Z=\text{Ker\thinspace}(\varphi)$ and 
$R=\text{Ker\thinspace}(\psi)$. Hence, instead of  
\eqref{eq:amalgam-projections}, we use the following notation:
\begin{equation}
\mathscr H=:H{\prescript{Z}{}\times_{L}^{R}}K~ \label{eq:amalg}.
\end{equation}

\vs
{\bf (b) Equivariant notation.} Below $\mathcal G$ stands for a compact Lie group.
For a subgroup $H$ of $\mathcal G$, 
denote by $N(H)$ the
normalizer of $H$ in $\mathcal G$ and by $W(H)=N(H)/H$ the Weyl group of $H$.  The symbol $(H)$ stands for the conjugacy class of $H$ in $\mathcal G$. 
Put $\Phi(\mathcal G):=\{(H): H\le \mathcal G\}$.
The set $\Phi (\mathcal G)$ has a natural partial order defined by 
$(H)\leq (K)$ iff $\exists g\in \mathcal G\;\;gHg^{-1}\leq K$. 
Put $\Phi_0 (\mathcal G):= \{ (H) \in \Phi(\mathcal G) \; : \; \text{$W(H)$  is finite}\}$.

For a $\mathcal G$-space $X$ and $x\in X$, denote by
$\mathcal G_{x} :=\{g\in \mathcal G:gx=x\}$  the {\it isotropy group}  of $x$
and call $(\mathcal G_{x})$   the {\it orbit type} of $x\in X$. Put $\Phi(\mathcal G,X) := \{(H) \in \Phi_0(\mathcal G) \; : \; 
(H) = (\mathcal G_x) \; \text{for some $x \in X$}\}$ and  $\Phi_0(\mathcal G,X):= \Phi(\mathcal G,X) \cap \Phi_0(\mathcal G)$. For a subgroup $H\leq \mathcal G$, the subspace $
X^{H} :=\{x\in X:\mathcal G_{x}\geq H\}$ is called the {\it $H$-fixed-point subspace} of $X$. If $Y$ is another $\mathcal G$-space, then a continuous map $f : X \to Y$ is called {\it equivariant} if $f(gx) = gf(x)$ for each $x \in X$ and $g \in \mathcal G$. 
Let $V$ be a finite-dimensional  $\mathcal G$-representation (without loss of generality, orthogonal).
Then, $V$  decomposes into a direct sum 
\begin{equation}
V=V_{0}\oplus V_{1}\oplus \dots \oplus V_{r},  \label{eq:Giso}
\end{equation}
where each component $V_{i}$ is {\it modeled} on the
irreducible $\mathcal G$-representation $\mathcal{V}_{i}$, $i=0,1,2,\dots ,r$, that is, $V_{i}$  contains all the irreducible subrepresentations of $V$
equivalent to $\mathcal{V}_{i}$. Decomposition  \eqref{eq:Giso}  is called  $\mathcal G$\textit{-isotypic  decomposition of} $V$.
\vs 
{\bf (b) Axioms of Equivariant Brouwer Degree.} Denote by  $\mathcal{M}^{\mathcal G}$ the set of all admissible $\mathcal G$-pairs and let $A(\mathcal G)$ stand for the Burnside ring of $\mathcal G$ (see Introduction, item (b)). The following result (cf.  \cite{AED}) can be considered as an axiomatic definition of the {\it $\mathcal G$-equivariant Brouwer degree}.

\begin{theorem}
\label{thm:GpropDeg} There exists a unique map $\mathcal G\mbox{\rm -}\deg:\mathcal{M}
^{\mathcal G}\to A(\mathcal G)$, which assigns to every admissible $\mathcal G$-pair $(f,\Omega)$ an
element $\gdeg(f,\Omega)\in A(\mathcal G)$
\begin{equation}
\label{eq:G-deg0}\mathcal G\mbox{\rm -}\deg(f,\Omega)=\sum_{(H)}%
{n_{H}(H)}= n_{H_{1}}(H_{1})+\dots+n_{H_{m}}(H_{m}),
\end{equation}
satisfying the following properties:

\begin{itemize}
\item[] \textbf{(Existence)} If $\mathcal G\mbox{\rm -}\deg(f,\Omega)\ne
0$, i.e., $n_{H_{i}}\neq0$ for some $i$ in \eqref{eq:G-deg0}, then there
exists $x\in\Omega$ such that $f(x)=0$ and $(\mathcal G_{x})\geq(H_{i})$.

\item[] \textbf{(Additivity)} Let $\Omega_{1}$ and $\Omega_{2}$
be two disjoint open $\mathcal G$-invariant subsets of $\Omega$ such that
$f^{-1}(0)\cap\Omega\subset\Omega_{1}\cup\Omega_{2}$. Then,
\begin{align*}
\mathcal G\mbox{\rm -}\deg(f,\Omega)=\mathcal G\mbox{\rm -}\deg(f,\Omega_{1})+\mathcal G\mbox{\rm -}\deg
(f,\Omega_{2}).
\end{align*}

\item[] \textbf{(Homotopy)} If $h:[0,1]\times V\to V$ is an
$\Omega$-admissible $\mathcal G$-homotopy, then
\begin{align*}
\mathcal G\mbox{\rm -}\deg(h_{t},\Omega)=\mathrm{constant}.
\end{align*}

\item[] \textbf{(Normalization)} Let $\Omega$ be a $G$-invariant
open bounded neighborhood of $0$ in $V$. Then,
\begin{align*}
\mathcal G\mbox{\rm -}\deg(\id,\Omega)=(\mathcal G).
\end{align*}

\item[] \textbf{(Multiplicativity)} For any $(f_{1},\Omega
_{1}),(f_{2},\Omega_{2})\in\mathcal{M} ^{\mathcal G}$,
\begin{align*}
\mathcal G\mbox{\rm -}\deg(f_{1}\times f_{2},\Omega_{1}\times\Omega_{2})=
\mathcal G\mbox{\rm -}\deg(f_{1},\Omega_{1})\cdot \mathcal G\mbox{\rm -}\deg(f_{2},\Omega_{2}),
\end{align*}
where the multiplication `$\cdot$' is taken in the Burnside ring $A(\mathcal G )$.

\item[] \textbf{(Recurrence Formula)} For an admissible $\mathcal G$-pair
$(f,\Omega)$, the $\mathcal G$-degree \eqref{eq:G-deg0} can be computed using the
following Recurrence Formula:
\begin{equation}
\label{eq:RF-0}n_{H}=\frac{\deg(f^{H},\Omega^{H})- \sum_{(K)>(H)}{n_{K}\,
n(H,K)\, \left|  W(K)\right|  }}{\left|  W(H)\right|  },
\end{equation}
where $\left|  X\right|  $ stands for the number of elements in the set $X$
and $\deg(f^{H},\Omega^{H})$ is the Brouwer degree of the map $f^{H}%
:=f|_{V^{H}}$ on the set $\Omega^{H}\subset V^{H}$.
\end{itemize}
\end{theorem}

The $\gdeg(f,\Omega)$ is 
 called the {\it $\mathcal G$%
-equivariant  Brouwer degree of $f$ in $\Omega$}.

%-

\vs
{\bf (c) Computation of Brouwer equivariant degree.} Put $B(V):=\left\{  x\in V:\left|  x\right|  <1\right\}  $. For each
irreducible $\mathcal G$-representation $\mathcal{V} _{i}$, $i=0,1,2,\dots$, define
\begin{align*}
\deg_{\mathcal{V}_{i}}:=\mathcal G\mbox{\rm -}\deg(-\id,B(\mathcal{V} _{i})),
\end{align*}
and call it  
the \emph{basic degree}.

Consider a $\mathcal G$-equivariant linear isomorphism $T:V\to V$ and assume that $V$
has a $\mathcal G$-isotypic  decomposition \eqref{eq:Giso}. Then, by the
Multiplicativity property,
\begin{equation}
\label{eq:prod-prop}\mathcal G\mbox{\rm -}\deg(T,B(V))=\prod_{i=0}^{r}\mathcal G\mbox{\rm -}\deg
(T_{i},B(V_{i}))= \prod_{i=0}^{r}\prod_{\mu\in\sigma_{-}(T)} \left(
\deg_{\mathcal{V} _{i}}\right)  ^{m_{i}(\mu)}%
\end{equation}
where $T_{i}=T|_{V_{i}}$ and $\sigma_{-}(T)$ denotes the real negative
spectrum of $T$, i.e., $\sigma_{-}(T)=\left\{  \mu\in\sigma(T):\mu<0\right\}
$. \vskip.3cm

Notice that the basic degrees can be effectively computed from \eqref{eq:RF-0}: 
\begin{align*}
\deg_{\mathcal{V} _{i}}=\sum_{(H)}n_{H}(H),
\end{align*}
where 
\begin{equation}
\label{eq:bdeg-nL}n_{H}=\frac{(-1)^{\dim\mathcal{V} _{i}^{H}}- \sum
_{H<K}{n_{K}\, n(H,K)\, \left|  W(K)\right|  }}{\left|  W(H)\right|  }.
\end{equation}

\end{document}